\documentclass{amsart}

\usepackage[utf8]{inputenc}

\usepackage{amssymb,amsfonts,color,comment,amsmath,verbatim,lscape,url}
\usepackage[all,arc]{xy}
\usepackage{enumerate}
\usepackage{mathrsfs}
\usepackage{stmaryrd}

\definecolor{todo}{rgb}{1,0,0}

\newtheorem{thm}{Theorem}[section]
\newtheorem*{thm*}{Theorem}
\newtheorem{cor}[thm]{Corollary}
\newtheorem{prop}[thm]{Proposition}

\theoremstyle{definition}
\newtheorem{defn}[thm]{Definition}
\newtheorem{defns}[thm]{Definitions}
\newtheorem{con}[thm]{Construction}

\newtheorem{ex}[thm]{Example}

\theoremstyle{remark}
\newtheorem{rmk}[thm]{Remark}

\makeatletter
\let\c@equation\c@thm
\makeatother

\let\oldmarginpar\marginpar
\renewcommand\marginpar[1]{\-\oldmarginpar[\raggedleft\footnotesize #1]%
{\raggedright\footnotesize #1}}

\newcommand{\Fun}{\mathrm{Fun}}

\newcommand{\Sp}{\mathrm{Sp}}

\newcommand{\End}{\mathrm{End}}

\newcommand{\Hom}{\mathrm{Hom}}
\newcommand{\colim}{\mathrm{colim}}

\newcommand{\fib}{\mathrm{fib}}
\newcommand{\tfib}{\mathrm{tfib}}

\newcommand{\cC}{\mathcal{C}}

\newcommand{\cR}{\mathcal{R}}

\newcommand{\Z}{\mathbb{Z}}
\newcommand{\N}{\mathbb{N}}

\newcommand{\id}{\mathrm{id}}

\makeatletter
\newcommand*{\doublerightarrow}[2]{\mathrel{
  \settowidth{\@tempdima}{$\scriptstyle#1$}
  \settowidth{\@tempdimb}{$\scriptstyle#2$}
  \ifdim\@tempdimb>\@tempdima \@tempdima=\@tempdimb\fi
  \mathop{\vcenter{
    \offinterlineskip\ialign{\hbox to\dimexpr\@tempdima+1em{##}\cr
    \rightarrowfill\cr\noalign{\kern.5ex}
    \rightarrowfill\cr}}}\limits^{\!#1}_{\!#2}}}
\newcommand*{\triplerightarrow}[1]{\mathrel{
  \settowidth{\@tempdima}{$\scriptstyle#1$}
  \mathop{\vcenter{
    \offinterlineskip\ialign{\hbox to\dimexpr\@tempdima+1em{##}\cr
    \rightarrowfill\cr\noalign{\kern.5ex}
    \rightarrowfill\cr\noalign{\kern.5ex}
    \rightarrowfill\cr}}}\limits^{\!#1}}}
\makeatother

\newcommand{\cF}{\mathcal F}
\newcommand{\cG}{\mathcal G}

\newcommand{\cP}{\mathcal P}
\newcommand{\cQ}{\mathcal Q}

\newcommand{\NE}{_{\neq\varnothing}}

\newcommand{\Q}{\mathbb Q}

\newcommand{\cX}{\mathcal{X}}
\newcommand{\cY}{\mathcal{Y}}
\newcommand{\cZ}{\mathcal{Z}}

\newcommand{\inz}{\mathrm{in}_{\varnothing}}

\newcommand{\Spd}[2]{{\Sp}_{#1}^{#2}}

\makeatletter
\ifx\SK@label\undefined\let\SK@label\label\fi
 \let\your@thm\@thm
 \def\@thm#1#2#3{\gdef\currthmtype{#3}\your@thm{#1}{#2}{#3}}
 \def\mylabel#1{{\let\your@currentlabel\@currentlabel\def\@currentlabel
  {\currthmtype~\your@currentlabel}
 \SK@label{#1@}}\label{#1}}
 \def\myref#1{\ref{#1@}}
\makeatother

\title{Chromatic fracture cubes}
\author{Omar Antol\'{\i}n-Camarena}
\address{Department of Mathematics, Harvard University,
Cambridge, MA \ 02138}
\email{oantolin@math.harvard.edu}
\author{Tobias Barthel}
\address{Max Planck Institute for Mathematics, Bonn, Germany}
\email{tbarthel@mpim-bonn.mpg.de}
\date{\today}

\begin{document}

\begin{abstract}
In this note, we construct a general form of the chromatic fracture cube, using a convenient characterization of the total homotopy fiber, and deduce a decomposition of the $E(n)$-local stable homotopy category.
\end{abstract}

\maketitle

\tableofcontents

\section{Introduction}

The chromatic fracture square can be interpreted as the chromatic analogue of the arithmetic pullback square
\[\xymatrix{\Z \ar[r] \ar[d] & \prod_p \Z_p \ar[d] \\
\Q \ar[r] & \Q \otimes \prod_p \Z_p}\]
and thus expresses a fundamental local-to-global principle in stable homotopy theory. As such, it has been repeatedly used to first study problems in an appropriate local context and to then reassemble the results. Most noticeably, this approach is used in the construction and study of the spectrum of topological modular forms, see \cite{behrenstmfnotes}.

The goal of this note is to construct a higher dimensional version of the chromatic fracture square, well-known to the experts, which allows to explicitly decompose the $E(n)$-localization of a spectrum into its $K(t)$-local pieces for $0\le t \le n$. Moreover, we obtain a decomposition of the $E(n)$-local category into diagram categories of $K(t)$-local categories, for $0\le t \le n$. In fact, we work in a slightly more abstract setting, see \myref{chromaticcube}, \myref{cubecat}, and \myref{decomposition} for precise statements of our results. 

The proof uses the characterization of the total homotopy fiber of an $n$-cube as a right adjoint; since we do not know of a published reference for this fact, we include the argument. This then allows to easily deduce the iterative construction of the total homotopy fiber, as can be found for instance in \cite{goodwillie2}. 

\subsection*{Acknowledgments} 

We thank Rune Haugseng for a helpful conversation about coCartesian fibrations. The second author would also like to thank the Max Planck Institute for Mathematics for its hospitality.

\section{Cubical homotopy preliminaries}

Let $\cC$ be a pointed $\infty$-category with finite limits, and let $\Fun(\cP(T),\cC)$ be the category of $n$-cubes in $\cC$. Here, $T$ is a fixed set with $n$ elements and $\cP(T)$ is its poset of subsets, ordered by inclusion. We will also use $\cP\NE(T)$ for the poset of non-empty subsets of $T$. If $\cF \in \Fun(\cP(T),\cC)$, then the total homotopy fiber $\tfib(\cF)$ of $\cF$ is the fiber of the natural map $f\colon \cF_{\varnothing} \to \lim \left( \cF|_{\cP\NE(T)} \right) = \lim_{S \ne \varnothing} \cF(S)$. Recall that $\cF$ is said to be Cartesian if $f$ is an equivalence; this implies that $\tfib(\cF)$ is contractible. If $\cC$ is stable the converse is true: $\tfib(\cF)$ being contractible implies that $\cF$ is Cartesian. 

\begin{prop}
The total (homotopy) fiber of a cubical diagram is right adjoint to the functor
\[\xymatrix{\cC \ar[r]^-{\inz} &\Fun(\cP(T),\cC)}\]
sending an element $X \in \cC$ to the $n$-cube with initial vertex $X$ and the terminal object $\ast$ elsewhere. 
\end{prop}

\begin{proof}
Let $\tilde{\cF}$ be an $n$-cube that is a limit cone over the diagram $\cF\NE := \cF|_{\cP\NE(T)}$; note that there is a canonical map $\cF \to \tilde{\cF}$. We thus get for $X \in \cC$:
\begin{align*}
\Hom(X,\tfib(\cF)) & =  \Hom(\inz(X),\inz(\tfib(\cF))) \\
& =  \Hom(\inz(X), \fib(\cF \to \tilde{\cF})) \\
& =  \fib(\Hom(\inz(X),\cF) \to \Hom(\inz(X), \tilde{\cF})) \\
& =  \Hom(\inz(X),\cF),
\end{align*}
because, as we will show, $\Hom(\inz(X), \tilde{\cF})$ is contractible.

Indeed, letting $T_+:=\{\ast\} \cup T$, we can identify
$\Hom(\inz(X),\cF)$ with the space of $(n+1)$-cubes $\cG \in
\Fun(\cP(T_+), \cC)$, such that $\cG$ restricts to $\inz(X)$
on $\{S \subseteq T_+ : \ast \notin S\}$ and restricts to
$\tilde{\cF}$ on $\{S \subseteq T_+ : \ast \in S\}$. Now, since
$\tilde{\cF}$ is Cartesian, these $(n+1)$-cubes $\cG$ are determined
up to contractible choice by their restrictions to $\{S
\subseteq T_+ : S \neq \{\ast\} \}$; and the space of these possible restrictions is contractible by the presence of the terminal objects in $\inz(X)$.
\end{proof}

As a consequence of this proposition we obtain an easy proof of the fact that the total homotopy fiber can be computed by taking fibers in all the edges of a cube in a fixed direction and then taking total fibers of the resulting cube. More precisely, let $\cF \in \Fun(\cP(T),\cC)$ be an $n$-cube and $T' \subseteq T$. For any $S' \in \cP(T')$, we can consider the cube $\cF_{T\setminus T',S'} \in \Fun(\cP(T\setminus T'),\cC)$ given by $S \mapsto \cF(S \cup S')$. This yields a new cube $\tfib_{T'}(\cF) \in \Fun(\cP(T'),\cC)$ whose value on $S' \subseteq T'$ is $\tfib(\cF_{T\setminus T',S'})$.

\begin{cor}\mylabel{tfibdecomposition}
With notation as above, there is a natural equivalence $\tfib(\cF) = \tfib(\tfib_{T'}(\cF))$ for any $T' \subseteq T$
\end{cor}
\begin{proof}
Observe that $\inz\colon\cC \to \Fun(\cP(T),\cC)$ can be decomposed as follows
\[\xymatrix{\cC \ar[r]^-{\inz} & \Fun(\cP(T'),\cC) \ar[r]^-{\inz} & \Fun(\cP(T\setminus T'),\Fun(\cP(T'),\cC)) \ar[r]^-{\sim} & \Fun(\cP(T),\cC),}\]
so the same is true for the right adjoint $\tfib$. 
\end{proof}

\begin{cor}\mylabel{fibcart}
If $\cC$ is stable and $\cF \in \Fun(\cP(T),\cC)$ is a Cartesian $n$-cube and $T'=\{t\} \subseteq T$ a subset of size $1$, then the fiber of $(n-1)$-cubes $$\fib(\cF_{T \setminus \{t\}, \varnothing} \to \cF_{T \setminus \{t\}, \{t\}})$$ is also Cartesian.
\end{cor}
\begin{proof}
This follows immediately from \myref{tfibdecomposition} and the fact that a cube in a stable category is Cartesian if and only if its total homotopy fiber is contractible.
\end{proof}

In the same spirit, but moving away from total fibers, we give a
formula for inductively computing limits of partial $n$-cubes $\cG\colon
\cP\NE(T) \to \cC$.

\begin{prop}\mylabel{pcubelim}
  Let $\cG : \cP\NE(T) \to \cC$ be a partial cube in $\cC$. Let $t \in
  T$ be arbitrary and set $T' = T \setminus \{t\}$. Then there is a
  pullback square:
  \[\xymatrix{
    \lim_{S \in \cP\NE(T)} \cG(S) \ar[r] \ar[d] &
    \lim_{S \in \cP\NE(T')} \cG(S) \ar[d]\\
    \cG(\{t\}) \ar[r] &
    \lim_{S \in \cP\NE(T')} \cG(\{t\} \cup S) }\]
\end{prop}

\begin{proof}
  Let $\cQ = \cP\NE(\{a,b\}) \times \cP\NE(T')$ and $F\colon \cQ \to
  \cP\NE(T)$ be the map of posets defined by $F(\{a\},I) = \{t\}$,
  $F(\{b\},I) = I$ and $F(\{a,b\},I) = \{t\} \cup I$. We will show
  that $F$ is (homotopy) initial. It suffices to check that for each
  $I \in \cP\NE(T)$, the comma category $F \downarrow I$ is
  contractible. In all cases the comma category is a subposet of
  $\cQ$:
  \begin{itemize}
  \item If $t \notin I$, $F \downarrow I$ is the subposet $\{
    (\{b\},J) \in \cQ : J \subseteq I \}$. This is contractible
    because it has a largest element, namely $(\{b\},I)$.
  \item If $t \in I$, $F \downarrow I$ is the subposet $\cR = \{(K,J)
    \in \cQ: K =\{a\} \textbf{ or } J \subseteq I\}$. We can write
    $\cR$ as a pushout of posets:
    \[\xymatrix{
      \{\{a\}\} \times \cP\NE(I) \ar[r] \ar[d] &
      \{\{a\}\} \times \cP\NE(T') \ar[d] \\
      \cP\NE(\{a,b\}) \times \cP\NE(I) \ar[r] &
      \cR} \]
    The geometric realizations of the three posets besides $\cR$ are
    cubes (of dimensions $|I|$, $n-1$ and $2+|I|$) and the top
    horizontal and left vertical maps realize to a face inclusion.
    This shows $\cR$ is contractible.
  \end{itemize}

  Since $F$ is initial, $\lim \cG$ can be computed as $\lim(\cG \circ
  F)$. This limit we compute as an iterated limit:
  $$\lim_{(K,J)\in\cQ} \cG(F(K,J)) = \lim_{K \in \cP\NE(\{a,b\})} \left(\lim_{J \in \cP\NE(T')} \cG(F(K,J))\right).$$

  To conclude we identify the three terms in the pullback with the ones
  in the statement of the proposition:

  \begin{enumerate}
  \item $\lim_{J \in \cP\NE(T')} \cG(F(\{a\}, J)) = \lim_{J \in
      \cP\NE(T')} \cG(\{1\}) \cong \cG(\{1\})$, since the indexing
    category $\cP\NE(T')$ is contractible.
  \item $\lim_{J \in \cP\NE(T')} \cG(F(\{b\}, J)) = \lim_{J \in
      \cP\NE(T')} \cG(J)$.
  \item $\lim_{J \in \cP\NE(T')} \cG(F(\{a,b\}, J)) = \lim_{J \in
      \cP\NE(T')} \cG(\{1\} \cup J)$.
  \end{enumerate}
\end{proof}

\section{The chromatic fracture cube}\label{sec:cfc}

Let $\Sp$ be the stable $\infty$-category of spectra \cite{HA}, and denote Bousfield localization \cite{Bousfieldlocalization} at a spectrum $E$ by $L_E$, with  associated category $\Spd{E}{}$ of $E$-local spectra. Two spectra $E$ and $F$ are said to be Bousfield equivalent if $L_E = L_F$; in this case, we write $\langle E \rangle = \langle F \rangle $. An endofunctor on $\Sp$ will be called $E$-local if it takes values in the category of $E$-local spectra. If we are given a collection $\{F(1), F(2), \ldots, F(n)\}$ of spectra, for any set $S = \{i_1,\ldots,i_k\}$ with $1 \le i_1 < i_2 < \cdots < i_k \le n$, we write $L_S$ for the composite $L_{F(i_1)} \ldots L_{F(i_k)}$. The general form of the chromatic fracture cube takes the following form, generalizing the fracture square in \cite{bauernotes}. 

\begin{con}\mylabel{cubecons}
  Suppose $\{F(1),\ldots,F(n)\}$ is any collection of spectra. We
  inductively define an $n$-cube $\cF : \cP(\{1,\ldots,n\})
  \to \End(\Sp)$, whose vertices will turn out to be given by $S
  \mapsto L_S$, as follows:
  \begin{itemize}
  \item If $n=1$, the cube is simply the natural localization morphism
    $\id \to L_{F(i)}$.
  \item For $n>1$, inductively construct the cube $\cF'$ on
    $\cP(\{2,\ldots,n\})$, and get the full $n$-cube $\cF$ as the
    morphism of $(n-1)$-cubes $\cF' \to L_{F(1)} \cF'$ given by the
    naturality of the localization morphisms $\id \to L_{F(1)}$.
  \end{itemize}
\end{con}

\begin{prop}\mylabel{chromaticcube}
  Suppose $\{F(1),\ldots,F(n)\}$ is a collection of spectra such that
  $L_{F(j)} L_{F(i)}=0$ for all $j>i$. If $\cF \in
  \Fun(\cP(\{1,\ldots,n\}),\End(\Sp))$ is the $n$-cube given
  by \myref{cubecons} then
  \[\lim_{S \ne \varnothing} \cF(S) = L_E,\]
  where $E$ is any spectrum Bousfield equivalent to $F(1) \oplus
  \ldots \oplus F(n)$.
\end{prop}

\begin{proof}
Let $P = \lim_{S \ne \varnothing} \cF(S)$ be the limit with legs $f_i\colon P \to L_{F(i)}$. The localization maps $\eta_i\colon L_E \to L_{F(i)}L_E = L_{F(i)}$ induce a natural map from $L_E$ to this limit,
\[\eta\colon L_E \longrightarrow P.\]
Since $F(i)$-locals are clearly $E$-local and locality is preserved under limits, $P$ is $E$-local. It therefore suffices to show that $\eta$ is an $F(i)$-equivalence for all $1 \le i \le n$. To this end, fix $i$ and consider the commutative triangle
\[\xymatrix{L_E \ar[r]^{\eta} \ar[rd]_{\eta_i} & P \ar[d]^{f_i} \\
& L_{F(i)}.}\]
Because $\eta_i$ is an $F(i)$-equivalence by definition, we only need to show that so is $f_i\colon P \to L_{F(i)}$. To this end, apply $L_{F_i}$ to the limit cube $\tilde{\cF}$ with $P$ as initial vertex, $\tilde{\cF}(S) = \cF(S)$ for $S\ne \varnothing$, and the natural maps. This yields a cube $L_{F_i}\tilde{\cF} \in \Fun(\cP(\{1,\ldots,n\}),\End(\Sp))$ with three properties:

\begin{itemize}
\item It is again Cartesian as $L_{F_i}$ is exact.
\item $L_{F(i)}\tilde{\cF} (S) = 0$ whenever $S$ contains an element
  smaller than $i$, because $L_{F(i)}L_{F(k)} = 0$ for $k<i$.
\item For $S \neq\varnothing$ and $i\notin S$, the edges from $S$ to
  $S \cup \{i\}$ are all equivalences since either both vertices are
  $0$ by the previous point or, if $\min(S)<i$, the edge $\cF(S) \to
  \cF(S \cup \{i\})$ of $\cF$ is equivalent to the localization map
  $\cF(S) \to L_{F(i)} \cF(S)$ by construction, so applying $L_{F(i)}$
  will yield an equivalence.
\end{itemize}

By \myref{fibcart}, taking fibers in the $\{i\}$ direction produces a Cartesian $(n-1)$-cube, which by the third item above is actually just $\inz(\fib(L_{F(i)}P \to L_{F(i)}))$. Therefore, $L_{F(i)}P \cong L_{F(i)}$ and the claim follows.
\end{proof}

\begin{ex}
The Morava $K$-theories $K(0),\ldots,K(n)$ satisfy the conditions of \myref{chromaticcube}, since $K(i) \otimes K(j) =0$ whenever $i\ne j$. In particular, if $i<j$, then there is a pullback square
\[\xymatrix{L_{K(i) \oplus K(j)} \ar[r] \ar[d] & L_{K(j)} \ar[d]\\
L_{K(i)} \ar[r] & L_{K(i)}L_{K(j)}.}\]
Similarly, for $F(1) = E(n-1)$ and $F(2) = K(n)$ we get the usual chromatic fracture square
\[\xymatrix{L_{E(n)} \ar[r] \ar[d] & L_{K(n)} \ar[d]\\
L_{E(n-1)} \ar[r] & L_{E(n-1)}L_{K(n)}}\]
using the identity of Bousfield classes $\langle E(n-1) \oplus K(n) \rangle = \langle E(n) \rangle$.
\end{ex}

\section{A description of the category of local objects}

As in the previous section, let $\{F(1),\ldots,F(n)\}$ be a collection
of spectra such that $L_{F(j)} L_{F(i)}=0$ for all $j>i$, fixed for
the remainder of the section, and let $T = \{1,\ldots,n\}$. We will
inductively construct a category $\cC_S$ of partial cubes $\cP\NE(S)$
for $S \subset T$, and prove that $\cC_T$ is equivalent to $\Sp_E$ for
any spectrum $E$ with the same Bousfield class as $F(1) \oplus \cdots
\oplus F(n)$.

\begin{con}\label{cubecatcons}
  For $\varnothing \neq S \subset T$ we let $\cC_S$ be the full
  subcategory of the diagram category $\Fun(\cP\NE(S), \Sp)$ spanned
  by certain partial cubes $\cG$ chosen as follows:
  \begin{itemize}
  \item If $S=\{i\}$ is a singleton, we take all $\cG$ such that the
    unique value of $\cG$, namely $\cG(S)$, is $F(i)$-local.
  \item If $S = \{i\} \cup S'$ where $i = \min(S) \notin S'$, we take
    all $\cG$ such that:
    \begin{enumerate}
    \item $\cG' := \cG|_{\cP\NE(S')}$ belongs to
      $\cC_{S'}$, and
    \item if we think of $\cG|_{\cP(S)\setminus\{\varnothing, \{i\}\}}$
      as a morphism between diagrams of shape $\cP\NE(S')$, namely $\cG'
      \to \cG'|_{\{U \subseteq S : i \in U\}}$, this morphism is the
      natural morphism $\cG' \to L_{F(i)} \cG'$.
    \end{enumerate}
  \end{itemize}
\end{con}

\begin{rmk}
  From the definitions it is clear that for any $\cG \in \cC_T$ we
  have that:
  \begin{itemize}
  \item For any $S \subseteq T$, $\cG(S)$ is $F(\min(S))$-local.
  \item For every $k$, $\cG|_{\{S \subseteq T : \max(S) = k\}}$ is just the
    cube from \myref{cubecons} applied to $\cG(\{k\})$.
  \end{itemize}
\end{rmk}

This construction becomes much clearer with an example.

\begin{ex}\mylabel{ex:cube}
  The $\infty$-category $\cC_{\{1,2,3\}}$ is equivalent to the full
  subcategory of $\Fun(\cP\NE(\{1,2,3\}),\Sp)$ on the diagrams of the form:
  \[\xymatrix{& & & X_3 \ar[ld]^{\eta_2} \ar[dd]^{\eta_1} \\
    X_2 \ar[rr]^{f_{23}} \ar[dd]_{\eta_1} & & L_{F(2)} X_3 \ar[dd]^(0.3){\eta_1} \\
    & X_1 \ar'[r][rr]^(0.3){f_{13}} \ar[ld]_{f_{12}} & & L_{F(1)}X_3 \ar[ld]^{L_{F(1)}\eta_2} \\
    L_{F(1)} X_2 \ar[rr]^{L_{F(1)} f_{23}} & & L_{F(1)}L_{F(2)} X_3}\]
  where $X_i$ is $F(i)$-local for $i=1,2,3$ and the morphisms labeled
  $\eta_i$ are the natural maps $Y \to L_{F(i)}Y$. Notice that the
  diagram is determined just by $X_1,X_2,X_3$, $f_{12},f_{13},f_{23}$
  and a homotopy showing the bottom square commutes.

  As a simpler example, the $\infty$-category $\cC_{\{1,2\}}$ is the
  category of cospans of the form of the left face of the above
  partial cube. Those cospans are determined by $X_1,X_2$ and
  $f_{12}$.
\end{ex}

\begin{thm}\mylabel{cubecat}
  If $E$ is any spectrum Bousfield equivalent to $F(1) \oplus \cdots
  \oplus F(n)$, the $\infty$-category $\cC_T$ constructed above is
  equivalent to $\Sp_E$. Moreover, the following functors are
  mutually inverse equivalences:
  \begin{itemize}
  \item $\lim\colon \cC_T \to \Sp_E$ given by $\cG \mapsto \lim_{S \neq
      \varnothing} \cG(S)$, and
  \item $\cF\NE\colon \Sp_E \to \cC_T$ given by $X \mapsto
    \cF[X]|_{\cP\NE(T)}$ where $\cF$ is the $n$-cube of functors from
    \myref{cubecons} and $\cF[X]$ denotes the cube obtained
    by applying those functors to $X$.
  \end{itemize}
\end{thm}

\begin{proof}
  First of all let us show those functors are well defined. This is
  clear for $\cF\NE$ by construction. For $\lim$ it is because all
  spectra in the image of a given partial cube $\cG$ are $F(i)$-local
  for some $i$ and thus also $E$-local.

  Now we will show the functors are mutually inverse. First, for any
  $E$-local spectrum $X$, the canonical map $\lim
  \left(\cF\NE[X]\right) \to X$ (coming from the diagram $\cF[X]$,
  which is a cone over $\cF\NE[X]$) is an equivalence by
  \myref{chromaticcube}.
  
  For the other composite, let $\cG \in \cC_T$ and let $P = \lim \cG$.
  We must show that $\cF\NE(P) \cong \cG$. Extend $\cG$ to a Cartesian
  cube $\tilde\cG$ with $\tilde\cG(\varnothing) = P$. First we show
  that $L_{F(i)}P \cong \cG(\{i\})$, and moreover, that the
  $\varnothing \to \{i\}$ edge in the $n$-cube $\tilde\cG$ is the
  localization map $P \to L_{F(i)}P$ under this equivalence.

  To this end, apply $L_{F(i)}$ to $\tilde\cG$. For any $i \notin S
  \subseteq T$ consider the edge of $L_{F(i)}\tilde\cG$ from $S$ to $S
  \cup \{i\}$. There are three cases:

  \begin{itemize}
  \item If $S = \varnothing$, the edge has the form $L_{F(i)} P \to
    L_{F(i)} \cG(\{i\})$. This is the morphism we wish to show is an
    equivalence.
  \item If $\min(S) < i$, we have $L_{F(i)} \cG(S) = L_{F(i)} \cG(S
    \cup \{i\}) = 0$, because both $\cG(S)$ and $\cG(S \cup \{i\})$
    are $F(\min(S))$-local. Therefore these edges are equivalences.
  \item If $\min(S) > i$, then the edge $\cG(S) \to \cG(\{i\} \cup S)$
    is the natural map $\cG(S) \to L_{F(i)} \cG(S)$ because this edge
    is contained in $\cG|_{\{U \subset T : \max(U)=\max(S)\}}$ which
    is a cube obtained from \myref{cubecons} applied
    to $\cG(\{\max(S)\})$. This edge, of course, becomes an
    equivalence after applying $L_{F(i)}$.
  \end{itemize}

  Now the cube $L_{F(i)}\tilde\cG$ is Cartesian because $\tilde\cG$
  was and $L_{F(i)}$ is exact. This means that taking fibers in the
  direction of $\{i\}$ must lead to a Cartesian $(n-1)$-cube. By the
  above case analysis, that $(n-1)$-cube is simply $\inz(\fib(L_{F(i)}
  P \to L_{F(i)} \cG(\{i\})))$, from which we conclude that
  $\fib(L_{F(i)} P \to L_{F(i)} \cG(\{i\}))$ is $0$ and thus $L_{F(i)}
  P \to L_{F(i)} \cG(\{i\})$ is an equivalence, as desired.
  
  At this point we are close to showing that $\tilde\cG$ and $\cF(P)$
  are equivalent $n$-cubes: we have shown that they have equivalent
  objects at all vertices and also many of the maps agree, but we
  have not shown for example that the map $\cG(\{1\}) \to \cG(\{1,i\})$
  is $L_{F(1)}(P \to \cG(\{i\}))$.
  
  To conclude, consider taking fibers of $\tilde\cG$ in the $\{1\}$
  direction to get an $(n-1)$-cube $\tilde\cG'$. From the argument
  above and the exactness of $L_{F(1)}$, we know that this
  $(n-1)$-cube vanishes if we apply $L_{F(1)}$ to it, so that the
  $n$-cube $\tilde\cG$ when thought of as a map of $(n-1)$-cubes
  $\tilde\cG|_{\{S : 1 \notin S\}} \to \cG|_{\{S : 1 \in S\}}$ is just
  $L_{F(1)}$-localization.

  Now applying $L_{F(2)}$ to $\tilde\cG$, we get a Cartesian $n$-cube
  whose ``bottom'' face, $L_{F(2)}\tilde\cG|_{\{S : 1\in S\}}$,
  vanishes. This means the top face, $L_{F(2)}\tilde\cG|_{\{S :
    1\notin S\}}$, is also Cartesian and we can recursively apply the
  argument of the previous paragraph to conclude that $\tilde\cG \cong
  \cF(P)$.
\end{proof}

\begin{rmk}
The special case $n=2$ of the previous result appears as Remark 7 in \cite[Lecture 23]{chromotopynotes}.
\end{rmk}

\section{A decomposition of the category of local objects}

As in the previous section, let $\{F(1),\ldots,F(n)\}$ be a collection
of spectra such that $L_{F(j)} L_{F(i)}=0$ for all $j>i$, and let $T =
\{1,\ldots,n\}$. In this section we will describe a partial $n$-cube
of $\infty$-categories whose limit is $\Sp_E$ where $E$ is Bousfield
equivalent to $F(1) \oplus \cdots \oplus F(n)$. To do that we will
need some combinatorial preliminaries; an illustration of the construction can be found below in \myref{ex:catcube}.

\begin{defns}
  For $\varnothing \neq S \subseteq S' \subseteq T$, define
  $\alpha(S)$, $\beta(S,S')$ and $\theta_{S,S'}$:
  \begin{align*}
    \alpha(S) & = \{U \subseteq T : S \subseteq U, \min(S) = \min(U)\} \\
    \beta(S,S') & = \{V \subseteq [\min(S'), \min(S)) : S' \cap
    [\min(S'), \min(S)) \subseteq V\} \\
    \theta_{S,S'}\colon \alpha(S') & \to \beta(S,S') \times
    \alpha(S) \\
    U & \mapsto (U \cap [\min(S'), \min(S)), U \cap [\min(S),n])
  \end{align*}
  where we have repurposed traditional interval notation to denote
  intervals of integers.
\end{defns}

We will regard $\alpha(S)$ and $\beta(S,S')$ as posets, ordering them
by inclusion, which makes $\theta_{S,S'}$ a map of posets. Notice that
$\alpha(S)$ and $\beta(S,S')$ are isomorphic to posets of all subsets
of some set, so that diagrams of shape $\alpha(S)$ or $\beta(S,S')$
are cubical diagrams.

\begin{con}\mylabel{decompcons}
  We construct a partial $n$-cube of $\infty$-categories $\cG\colon
  \cP\NE(T) \to \mathrm{Cat}_\infty$ as follows:
  \begin{itemize}
  \item The vertices are given by $\cG(S) =
    \Spd{F(\min(S))}{\alpha(S)} := \Fun(\alpha(S), \Sp_{F(\min(S))})$.
  \item For $S \subset S'$, the functor $\cG(S \subseteq S') : \cG(S)
    \to \cG(S')$ is given by the composite
    \begin{align*}
      \Fun(\alpha(S),\Sp_{F(\min(S))}) \xrightarrow{\iota \circ -} & \Fun(\alpha(S), \Sp) \\
      \xrightarrow{\cF|_{\beta(S,S')} \circ -} & \Fun(\alpha(S), \Fun(\beta(S,S'), \Sp_{F(\min(S'))})) \\
      \xrightarrow{\cong} & \Fun(\beta(S,S') \times \alpha(S), \Sp_{F(\min(S'))}) \\
      \xrightarrow{- \circ \theta} & \Fun(\alpha(S'), \Sp_{F(\min(S'))})
    \end{align*}
  \end{itemize}
  where $\iota\colon \Spd{F(\min(S))}{} \to \Sp$ denotes the natural inclusion functor and  $\cF$ is essentially the cube of functors from
  \myref{cubecons} restricted to $\beta(S,S') \subseteq \cP(T)$: this
  restriction is a functor $\beta(S,S') \to \End(\Sp)$, which we can think
  of instead as a functor $\Sp \to \Fun(\beta(S,S'), \Sp)$ and then
  replace $\Sp$ in the target by $\Sp_{F(\min(S'))}$ since all $V \in
  \beta(S,S')$ satisfy $\min(V) = \min(S')$.
\end{con}

\begin{rmk}
  Unwinding the definitions in Constructions \ref{decompcons} and
  \ref{cubecons}, we see that for, $\varnothing \neq S \subseteq S'$,
  the functor $\cG(S) \to \cG(S')$ sends a cube $X\colon \alpha(S) \to
  \Sp_{F(\min(S))}$ to the cube $X'\colon \alpha(S') \to
  \Sp_{F(\min(S'))}$ given on vertices by $$X'(U) = L_{U \cap
    [\min(S'), \min(S))} X(U \cap [\min(S),n]).$$

  Notice that if $\min(S') < \min(S)$, the formula shows $X'(U)$ is
  $L_{F(\min(S'))}$-local, as it should be. Also, when $\min(S') =
  \min(S)$, there is no localization at all and $X'$ is simply the
  restriction of $X$ to the face $\alpha(S') \subseteq \alpha(S)$
  (this inclusion does not hold when $\min(S) \neq \min(S')$).
\end{rmk}

This definition also becomes much clearer with an example:

\begin{ex}\mylabel{ex:catcube}
  Let $n=3$. The diagram $\cG$ looks like:

  \[\xymatrix{& \Spd{E}{} \ar[rr] \ar'[d][dd] \ar[ld] & & \Spd{F(3)}{\{\underline{3}\}} \ar[dd] \ar[ld] \\
    \Spd{F(2)}{\{\underline{2}, \underline{23}\}} \ar[rr] \ar[dd] & & \Spd{F(2)}{\{\underline{23}\}} \ar[dd] \\
    & \Spd{F(1)}{\{\underline{1}, \underline{12}, \underline{13}, \underline{123}\}} \ar'[r][rr] \ar[ld] & & \Spd{F(1)}{\{\underline{13}, \underline{123}\}} \ar[ld] \\
    \Spd{F(1)}{\{\underline{12},\underline{123}\}} \ar[rr] & & \Spd{F(1)}{\{\underline{123}\}}}\]
  where we have used shorthand for the elements of the various
  $\alpha(S)$: $\underline{13}$ denotes the set $\{1,3\}$, for
  example.
  
  Though $\cG(\varnothing)$ is not defined, we have put $\Sp_{E}$ in
  that corner, since \myref{decomposition} will show that this
  produces a Cartesian cube for any $E$ which is Bousfield equivalent
  to $F(1) \oplus F(2) \oplus F(3)$.
 
  A square in $\Spd{F(1)}{\{\underline{1}, \underline{12},
    \underline{13}, \underline{123}\}}$ should be though of as being
  \[\xymatrix{L_{F(1)} X \ar[r] \ar[d] & L_{F(1)}L_{F(3)} X \ar[d] \\
    L_{F(1)}L_{F(2)} X \ar[r] & L_{F(1)}L_{F(2)}L_{F(3)} X}\]
  for some spectrum $X$, and the two functors out of
  $\Spd{F(1)}{\{\underline{1}, \underline{12}, \underline{13},
    \underline{123}\}}$ are projection to the faces
  \[\xymatrix{L_{F(1)}L_{F(2)}X \ar[r] & L_{F(1)}L_{F(2)}L_{F(3)}X}\]
  and
  \[\xymatrix{L_{F(1)}L_{F(3)}X \ar[r] & L_{F(1)}L_{F(2)}L_{F(3)}X}.\]

  The vertical functor $\Spd{F(3)}{\{\underline{3}\}} \to
  \Spd{F(1)}{\{\underline{13}, \underline{123}\}}$ sends a
  $F(3)$-local spectrum $X$ to the morphism $L_{F(1)} X \to L_{F(1)}
  L_{F(2)} X$ you get by applying $L_{F(1)}$ to the localization map
  $X \to L_{F(2)} X$.
\end{ex}

Our decomposition of the category of $\bigoplus_{i=1}^nF(i)$-local objects can now be stated as follows.

\begin{thm}\mylabel{decomposition}
If $E$ is any spectrum Bousfield equivalent to $F(1) \oplus \ldots \oplus F(n)$, then there is an equivalence of stable presentable $\infty$-categories
\[\xymatrix{\Spd{E}{} \ar[r]^-{\sim} & \lim_{S \ne \varnothing} \cG(S)}\]
where $\cG$ is given by \myref{decompcons}.
\end{thm}

\begin{proof}
For any $S \in T$, observe that there is a natural functor 
\[\xymatrix{\Phi_S\colon \cC_T \ar[r] & \Spd{F(\min(S))}{\alpha(S)},}\]
sending a partial $T$-cube $\cX$ to the cube given by restriction,
$\Phi_S(\cX) = \cX|_{\alpha(S)}$. The collection of these functors
induces a natural functor
\[\xymatrix{\Phi\colon \cC_T \ar[r] & \lim_{S \ne \varnothing} \cG(S)}\]
and it suffices to show that this is an equivalence by \myref{cubecat}.

We will now argue by induction on $n$, the case $n=1$ being trivial.
\myref{pcubelim} applied with $t=1$ shows that there is a pullback diagram
\[\xymatrix{\lim_{S \ne \varnothing} \cG(S) \ar[r] \ar[d] & \lim_{S \in \alpha(\underline{1}) \setminus \{\underline{1}\}} \cG(S\setminus \underline{1}) \ar[d] \\
\Spd{F(1)}{\alpha(\underline{1})} \ar[r] & \lim_{S \in \alpha(\underline{1})\setminus \{\underline{1}\}}\cG(S)}\]
(where again $\underline{1} = \{1\}$) and the inductive hypothesis
gives $\lim_{S \in \alpha(\underline{1}) \setminus \{\underline{1}\}}
\cG(S\setminus \underline{1}) = \cC_{\{2,\ldots,n\}}$. Also,
since all the edges in the bottom face of $\cG$ are 
restrictions:
\begin{align*}
\lim_{S \in \alpha(\underline{1})\setminus \{\underline{1}\}}\cG(S) & = \lim_{S \in \alpha(\underline{1})\setminus \{\underline{1}\}} \Spd{F(1)}{\alpha(S)} \\
& = \Spd{F(1)}{\colim_{S \in \alpha(\underline{1})\setminus \{\underline{1}\}} \alpha(S)} \\
& = \Spd{F(1)}{\alpha(\underline{1})\setminus \{\underline{1}\}},
\end{align*}
so it suffices to show that the following commutative square is a pullback:
\[\xymatrix{\cC_T \ar[r]^-p \ar[d]_-f & \cC_{\{2,\ldots,n\}} \ar[d]^{L_{F(1)}} \\
\Spd{F(1)}{\alpha(\underline{1})} \ar[r]_-q & \Spd{F(1)}{\alpha(\underline{1}) \setminus \{\underline{1}\}}.}\]

This is intuitively clear: objects of the pullback can be
described by giving a diagram $\cY \in \cC_{\{2,\ldots,n\}}$, a
diagram $\cX \in \Spd{F(1)}{\alpha(\underline{1})}$, and an
equivalence $\cX|_{\alpha(\underline{1})\setminus\{\underline{1}\}}
\to L_{F(1)} \cY$. That data clearly assembles to make a diagram in
$\cZ \in \cC_T$ with top (partial) face $\cZ|_{\cP(\{2,\ldots,n\})} =
\cY$ and bottom face $\cZ|_{\alpha(\underline{1})} = \cX$.

More formally, first note that it is easy to check that the horizontal arrows $p$ and $q$ in the above diagram are coCartesian fibrations and that the left vertical map $f$ preserves coCartesian morphisms. Therefore, we can make use of \cite[2.4.4.4]{htt} to reduce the claim to checking that, for every $\cY \in \cC_{\{2,\ldots,n\}}$, the fiber over $\cY$ in the horizontal direction are  equivalent via $f$, i.e., 
\[\xymatrix{(\cC_T)_{\cY} \ar[r]^-{\sim}_-{f} & (\Spd{F(1)}{\alpha(\underline{1})})_{L_{F(1)}\cY}.}\]

Let us first consider the case $n=2$. The fiber over $\cY = Y \in \Spd{F(2)}{}$ is given by the full subcategory of $\cC_{\{1,2\}}$ on object of shape 
\[\xymatrix{& Y \ar[d] \\
X \ar[r] & L_{F(1)}Y,}\]
which are thus determined by the bottom morphism $X \to L_{F(1)}Y$. This category is then easily seen to be equivalent to the fiber $(\Spd{F(1)}{\alpha(\underline{1})})_{L_{F(1)}Y}$, hence the claim holds for $n=2$. Now we can explain how to deduce the general case from here. We have a commutative diagram of fiber sequences:
\[\xymatrix{
(\cC_T)_{\cY}  \ar[r] \ar@/_3.5pc/@{-->}[ddd]_{\sim} \ar[d]^{\sim} & \cC_T \ar[r] \ar[d]^{\sim} \ar@/_3.5pc/@{-->}[ddd]_f &  \cC_{\{2,\ldots,n\}} \ar[d]_{\lim}^{\sim} \ar@/^3.5pc/@{-->}[ddd]^{L_{F(1)}} \\
(\Spd{E}{})_{\lim\cY} \ar[r] \ar[d]^{\sim} &  \Spd{E}{} \ar[r] \ar[d] &  \Spd{F(2) \oplus \cdots \oplus F(n)}{} \ar[d]_{L_{F(1)}} \\
(\Spd{F(1)}{\cP(\underline{1})})_{\lim L_{F(1)}\cY} \ar[r] & \Spd{F(1)}{\cP(\underline{1})} \ar[r]^{\mathrm{ev}_{\underline{1}}} &  \Spd{F(1)}{}\\
(\Spd{F(1)}{\alpha(\underline{1})})_{L_{F(1)}\cY} \ar[r] \ar[u]_{\sim} & \Spd{F(1)}{\alpha(\underline{1})} \ar[r] \ar[u] & \Spd{F(1)}{\alpha(\underline{1}) \setminus \{\underline{1}\}}. \ar[u]^{\lim}}\]
where the top right square is a pullback by \myref{cubecat}, the middle right one is by the induction hypothesis applied to the pair $(F(1), F(2) \oplus \cdots \oplus F(n))$, and the the bottom right one is by construction. The claim follows. 
\end{proof}

\bibliographystyle{alpha}
\bibliography{bibliography}

\end{document}